\documentclass{elsarticle}
\usepackage{amsmath,amssymb,amsfonts,graphics}

\usepackage{amssymb,graphicx,amscd,accents,array,url}
\usepackage[usenames,dvipsnames]{color}

\usepackage{amsmath}
\usepackage{latexsym}
\usepackage{mathrsfs}
\usepackage{amsthm}
\allowdisplaybreaks

\def\qed{\hbox{${\vcenter{\vbox{                        
   \hrule height 0.4pt\hbox{\vrule width 0.4pt height 6pt
   \kern5pt\vrule width 0.4pt}\hrule height 0.4pt}}}$}}

\newtheorem{theorem}{{\bf Theorem}}[section]

\newtheorem{proposition}{{\bf Proposition}}[section]
\newtheorem{remark}{{\bf Remark}}[section]

\newcommand{\reg}{p}
\renewcommand{\O}{\Omega}
\newcommand{\dx}{\,\mbox{d}x}

\def\wp{{w}_{\pi}}
\newcommand\E{{E}}

\newcommand\DD{\mathcal D}
\newcommand\Th{{\mathcal T}_h}
\newcommand\Eh{{\mathcal E}_h}

\newcommand\Pp{{\mathbb P}}

\newcommand{\calG}{ \mathcal{G}}

\newcommand\R{\mathbb{R}}

\numberwithin{equation}{section}

\def\PP{{\mathbb P}}  

\def\O{\Omega}

\def\Th{{\mathcal T}_h}
\def\E{K}

\def\R{{\mathbb R}}
\def\VV{{\boldsymbol \xi}}

\begin{document}

\title{Virtual Element Method for fourth order problems: $L^2-$estimates}

\author[uno]{Claudia Chinosi}
\author[due]{L. Donatella Marini}
\address[uno]{Dipartimento di Scienze e Innovazione Tecnologica, Universit\`a del Piemonte Orientale,  Viale Teresa Michel 11, 15100, Alessandria, Italy email: {chinosi@uniupo.it}}

\address[due]{Dipartimento di Matematica, Universit\`a di Pavia, and IMATI-CNR, Via Ferrata 1, 27100, Pavia, Italy email: 
{marini@imati.cnr.it}}

\begin{abstract}
We analyse the family of $C^1$-Virtual Elements introduced in \cite{Brezzi:Marini:plates} for fourth-order problems and prove optimal estimates in $L^2$ and in $H^1$ via classical duality arguments.
 \end{abstract}

\maketitle
\section{Introduction}
%
The Virtual Element Method (VEM), introduced in \cite{volley} and further developed in \cite{projectors}, can be seen as the extension of the Finite Element Method (FEM) to decompositions into almost arbitrary polygons and/or polyhedra. Since the first paper in 2013 (\cite{volley}) the Virtual Element approach has been applied to a number of applications: linear elasticity in two and three dimensions (\cite{VEM-elasticity} and \cite{Paulino-VEM}, respectively), general advection-diffusion-reaction problems, both in primal \cite{variable-primal} and in mixed form \cite{mixed-variables}, Helmholtz problem \cite{Helmo-PPR}, and plate bending problems  in the Kirchhoff-Love formulation \cite{Brezzi:Marini:plates}. In \cite{Brezzi:Marini:plates} a family of elements was constructed and analysed, showing the ductility of the approach to design $C^1$-elements. Optimal convergence rates were proved in the energy norm, i.e., in $H^2$. Namely, order $k-1$, with $k\ge 2$ whenever the discrete space $V_h$ contains locally polynomials of degree $k$.
In the present paper we prove optimal estimates also in $H^1$ and in $L^2$, obtained via classical duality arguments, and we provide numerical results confirming the theoretical estimates. 

We point out that the use of $C^1$-approximations is of interest not only for plate bending problems, although the family of elements we are dealing with was originally introduced having in mind plates. In many other applications the presence of fourth order operators calls for higher continuity. For example, Cahn-Hilliard  equation for phase separation, or Navier-Stokes equations in the stream-vorticity formulation contain the biharmonic operator, exactly as in plate bending problems, which we will refer to throughout the paper.

An outline of the  paper is as follows. In Section \ref{PBcont} we state the continuous problem and fix some notation.  In Section \ref{PBdiscr} we recall the VEM-approximation and the convergence result given in \cite{Brezzi:Marini:plates}. In particular, in Subsection \ref{load} we propose a different approximation of the loading term, more suited for deriving optimal estimates in $L^2$ and $H^1$. In Sections \ref{est-H1} and \ref{est-L2} we prove error estimates in $H^1$ and in $L^2$, respectively. Numerical results are presented in Section \ref{num-tests}, and a comparison with the classical Clough-Tocher and Reduced-Clough-Tocher finite elements  is carried out. 

Throughout the paper we shall use the common notation for the Sobolev spaces
$H^m(\DD)$ for $m$ a non-negative integer and $\DD$ an open bounded domain. In
particular (see e.g. \cite{Lions:Magenes:1968}, \cite{Ciarlet-78}) the $L^2(\DD)$ scalar product and norm
will be indicated by $(\cdot,\cdot)_{0,\DD}$ or $(\cdot,\cdot)_{\DD}$ and
$\|\cdot\|_{0,\DD}$ or $\|\cdot\|_{\DD}$, respectively. 
When $\DD\equiv\O$ the subscript $\DD$ will often be omitted. Finally, $\PP_k$ will denote the space of polynomials of degree $\le k$, with the convention that $\PP_{-1}=\{0\}$, and $C$ will denote a positive constant independent of the mesh size $h$.
\section{The Continuous Problem}\label{PBcont}
Let $\O\subset \R^2$ be a convex polygonal domain occupied by the plate, let $\Gamma$ be its boundary,
and let $f\in L^2(\O)$ be a transversal load acting on the plate.
The Kirchoff-Love model for thin plates (see e.g. \cite{Ciarlet-78}) corresponds to look for the transversal
displacement $w$, the solution of
\begin{equation}\label{biharm}
D \Delta^2\,w~=~f\qquad \mbox{in } \O,
\end{equation}
where $D=E t^3/12(1-\nu^2)$ is the bending rigidity, $t$  the thickness,
$E$  the Young modulus, and $\nu$ the Poisson's ratio.
Assuming for instance the plate to be clamped all over the boundary, equation
\eqref{biharm} is supplemented with the boundary conditions
\begin{equation}\label{BC}
w=\frac{\partial w}{\partial n}=0 \qquad \mbox{on } \Gamma.
\end{equation}
The variational formulation of \eqref{biharm}-\eqref{BC} is:
\begin{equation}\label{pb-cont}
\left\lbrace{
\begin{aligned}
&\mbox{Find } w\in V:=H^2_0(\O) \mbox{ solution of}\\
&a(w,v)=(f,v)_0\quad \forall v\in H^2_0(\O),
  \end{aligned}
  } \right.
  \end{equation}
where the energy bilinear form $a(\cdot,\cdot)$ is given by
  \begin{equation}\label{form-a}
  a(w,v)=D\Big[(1-\nu)\int_{\O}w_{/ij}v_{/ij} \dx +
\nu \int_{\O}\Delta w \Delta v \dx\Big].
  \end{equation}
   In \eqref{form-a} $v_{/i}=\partial v/\partial x_i,~i=1,2$,
and we used the summation
convention of repeated indices. 
   Setting $\|v\|_V:=|v|_{2,\O}$, it is easy to see that, thanks to the boundary conditions in $V$ and to the Poincar\'e inequality, this is indeed {\it a norm}
   on $V$. Moreover
   \begin{equation}\label{cont-a}
   \exists\, M>0\quad \mbox{such that }\quad a(u,v)\le M \|u\|_V\|v\|_V\quad u,v \in V,
   \end{equation}
   \begin{equation}\label{ellipt-a}
   \exists\, \alpha>0\quad \mbox{such that }\quad a(v,v)\ge \alpha \|v\|^2_V\quad v \in V.
   \end{equation}
   Hence, \eqref{pb-cont} has a unique solution, and (see, e.g. \cite{Lions:Magenes:1968})
  \begin{equation}
  \| w \|_V \le C \|f\|_{L^2(\O)}.
  \end{equation}

\section{Virtual Element discretization}\label{PBdiscr}
We recall the construction of the family of elements given in \cite{Brezzi:Marini:plates}, and the estimates there obtained.
The family of elements depends on three integer indices $(r,~s,~m)$,
related to the degree of accuracy $k\ge 2$ by:
 \begin{equation}\label{degrees-r-s-m}
  r=\max\{ 3,k\},\quad s=k-1,\quad m=k-4.
 \end{equation}
Let $\Th$ be a decomposition of $\O$ into polygons $\E$, and let $\Eh$ be the set of edges in $\Th$. We denote by $h_\E$ the diameter of $\E$, i.e., the maximum distance between any two vertices of $\E$.
On $\Th$ we make the following assumptions (see e.g. \cite{volley}):

\noindent
{\bf H1} there exists a fixed number $\rho_0>0$, independent of $\Th$, such that for every element $\E$ (with diameter $h_\E$) it holds 

\noindent
{\it i)} $\E$ is star-shaped with respect to all the points of a ball of radius $\rho_0\,h_\E$, and 

\noindent
{\it ii)} every edge $e$ of $\E$ has length $|e|\ge \rho_0\,h_\E$. 

\subsection{Definition of the discrete space $V_h$}
On a generic polygon $\E$ we define the local virtual element space as
\begin{equation*}\label{space-local}
V(\E):=\{v\in H^2(\E): ~v_{|e}\in\Pp_r(e),~\frac{\partial v}{\partial n}_{|e}\in \Pp_s(e)~\forall e\in \partial{\E},~\Delta^2 v\in \Pp_{m}(\E)\}.
\end{equation*}
Then the global space $V_h$ is given by
\begin{equation}\label{def:Vh}
V_h=\{v\in V:  v_{|\E}\in V(\E),~\forall \E \in \Th \}.
\end{equation}
A function in $V_h$ is uniquely identified by the following degrees of freedom:
\begin{eqnarray}
&&\!\!\!\!\!\!\!\bullet\mbox{ The value of }v(\VV)
 \quad \forall ~vertex~ \VV \label{vertices}
 \\
 &&\!\!\!\!\!\!\!\bullet
 \mbox{ The values of }  v_{/1}(\VV)\mbox{ and } v_{/2}(\VV)
 \quad \forall ~vertex~ \VV \label{gradvertices}\\
&&\!\!\!\!\!\!\!\bullet \mbox{ For } r>3,
\mbox{ the moments }\displaystyle{\int_e q(\xi)v(\xi){\rm d}\xi} \quad
\forall q\in  \Pp_{k-4}(e) , 
 ~ \forall e \in \Eh \label{edges}\\
&&\!\!\!\!\!\!\!\bullet\mbox{ For } s>1, \mbox{ the moments }
\displaystyle{\int_e q(\xi)v_{/n}(\xi){\rm d}\xi}
\quad \forall q\in \Pp_{s-2}(e) , 
~ \forall e \in \Eh  \label{normal}\\
&&\!\!\!\!\!\!\!\bullet\mbox{ For } m\ge 0, \mbox{ the moments }
\displaystyle{\int_{\E} q(x)v(x)\dx} \quad \forall q\in \Pp_m(\E) \quad \forall \E. 
\label{interior}
\end{eqnarray}
\begin{proposition}\label{dofbordo}In each element $\E$ the d.o.f. \eqref{vertices}--\eqref{interior} are unisolvent. Moreover, \eqref{vertices}, \eqref{gradvertices}, and \eqref{edges} uniquely determine a polynomial of degree $\le r$ on each edge of $\E$, the degrees of freedom \eqref{gradvertices} and \eqref{normal} uniquely determine a polynomial of degree $\le s$ on each edge of $\E$, and the d.o.f. \eqref{interior} are equivalent to prescribe $\Pi^{0}_m v$ in $\E$, where, for $m$ a nonnegative integer,
 \begin{equation}\label{defPm}
\Pi^{0}_m v \rm{ ~is ~the~ } L^2(\E)-\rm{projector ~operator~ onto~ the~ space ~}\Pp_m(\E).
\end{equation}
 \end{proposition}
\begin{remark}\label{interperr} We recall that our assumptions on $\Th$ allow to define, for every smooth enough $w$, an ``interpolant'' in $ V_h$ with the right interpolation properties. More precisely, if $g_i(w),~i=1,2,...\calG$ are the global d.o.f. in $V_h$, there exists a unique element $w_I\in V_h$ such that
 \begin{equation}\label{definterp}
 g_i(w-w_I)=0\qquad\forall i=1,2,....\calG.
 \end{equation}
 Moreover, by the usual Bramble-Hilbert technique (see e.g. \cite{Ciarlet-78}) and scaling arguments 
 (see e.g. \cite{Bo-Bre-For})  we can prove that
 \begin{equation}\label{errinterp}
 \|w-w_I\|_{s,\O}\le \,C\,h^{\beta-s}\,|w|_{\beta,\O}\qquad
s=0,1,2,\quad 3\le \beta\le k+1
 \end{equation}
(with $C>0$ independent of $h$) as in the usual Finite Element framework.
\end{remark}

\subsection{Construction of $a_h$} \label{costrah}
We need to define a symmetric discrete bilinear form which is stable and consistent. More precisely,  denoting by $a^{\E}_h(\cdot,\cdot)$ the restriction of $a_h(\cdot,\cdot)$ to a generic element $\E$, the following properties must be satisfied (see \cite{volley}).
For all $h$, and for all  ${\E}$ in $\Th$,  
\begin{itemize}
\item {\it k-Consistency}: $\forall p_k\in \Pp_k,\,\forall v\in V(\E)$
\begin{equation}\label{basic}
a_h^{\E}(p_k,v)=a^{\E}(p_k,v).
\end{equation}
\item {\it Stability}: $\exists$ two positive constants $\alpha_*$ and $\alpha^*$, independent of $h$ and of $\E$,  such that
    \begin{equation}\label{ahanda}
    \forall v\in V(\E)\qquad \alpha_*\,a^{\E}(v,v)\le a_h^{\E}(v,v)
    \le \alpha^*\,a^{\E}(v,v).
    \end{equation}
\end{itemize}
The symmetry of $a_h$, \eqref{ahanda}
and the continuity of $a^{\E}$ imply 
the continuity of $a^{\E}_h$:
\begin{equation}\label{CSahE}
\begin{aligned}
a_h^{\E}(u,v)&\le\,\Big(a_h^{\E}(u,u)\Big)^{1/2}\,\Big(a_h^{\E}(v,v)\Big)^{1/2}
\\
&\le\,\alpha^*\,M\,\|u\|_{2,\E}\,\|v\|_{2,\E}\quad \mbox{for all $u$ and  $v$ in $V(\E)$}.
\end{aligned}
\end{equation}
In turn, \eqref{ahanda} and \eqref{CSahE} easily imply
\begin{equation}\label{ahandaO}
    \forall v\in V_h\qquad \alpha_*\,a(v,v)\le a_h(v,v)
    \le \alpha^*\,a(v,v),
    \end{equation}
and
\begin{equation}\label{CSahEO}
a_h(u,v)\le\,\alpha^*\,M\,\|u\|_{V}\,\|v\|_{V}\quad \mbox{for all $u$ and  $v$ in $V_h$}.
\end{equation}
Next, let $\Pi^{\E}_{k}: {V(\E)}\longrightarrow \Pp_k(\E)\subset {V(\E)}$ be the operator defined as the solution of
\begin{equation}\label{def:Pi}
\left\lbrace{
\begin{aligned}
&a^{\E}(\Pi^{\E}_{k} \psi, q)=a^{\E}(\psi,q)\qquad \forall \psi\in {V(\E)},~\forall q\in \Pp_k(\E)\\
&\int_{\partial\E}(\Pi^{\E}_{k} \psi-\psi)=0,\quad \int_{\partial\E}\nabla (\Pi^{\E}_{k} \psi-\psi)=0.
\end{aligned}
} \right.
\end{equation} 
We note that for $v\in \Pp_k(\E)$ the first equation in \eqref{def:Pi} implies $(\Pi^{\E}_{k}v)_{/ij} =v_{/ij}$ for $i,j=1,2$, that joined with the second equation gives easily
\begin{equation}\label{uffa}
\Pi^{\E}_{k} v=v\quad \forall v\in \PP_k(\E).
\end{equation}
Hence, $\Pi^{\E}_{k}$ is a projector operator onto $\PP_k(\E)$.
Let then $S^{\E}(u,v)$ be a symmetric positive definite bilinear form, verifying
\begin{equation}\label{defStab}
 c_0\,a^{\E}(v,v)\le S^{\E}(v,v)
    \le c_1\,a^{\E}(v,v),\qquad \forall v\in {V(\E)} \mbox{ with } \Pi^{\E}_{k}v=0,
    \end{equation}
for some positive constants $c_0,~c_1$ independent of $\E$ and $h_{\E}$.  We refer to \cite{Brezzi:Marini:plates} for a precise choice of $S^{\E}(u,v)$. We just recall that $S^{\E}(u,v)$ can simply be taken as the euclidean scalar product associated to the degrees of freedom, properly scaled to satisfy \eqref{defStab}. Then set
\begin{equation}\label{def-ah}
a^{\E}_h(u,v):= a^{\E}(\Pi^{\E}_{k} u, \Pi^{\E}_{k} v)+ S^{\E}(u-\Pi^{\E}_{k} u,v-\Pi^{\E}_{k} v).
\end{equation}
Clearly the bilinear form \eqref{def-ah} verifies both the consistency property \eqref{basic} and the stability property \eqref{ahanda}.

We postpone the construction of the right-hand side, and recall the convergence result of \cite{Brezzi:Marini:plates}.
\begin{theorem}\label{teoconv} Under  assumptions {\bf H1} on the decomposition the discrete problem:
\begin{equation}\label{Vem-pb}
\left\lbrace{
\begin{aligned}
&\mbox{Find } w_h\in V_h \mbox{ solution of}\\
&a_h(w_h,v_h)=<f_h,v_h>\quad \forall v_h\in V_h
  \end{aligned}
  } \right.
  \end{equation}
has a unique solution $w_h$. Moreover, for every
approximation $w_I$ of $w$ in $V_h$ and for
every approximation $\wp$ of $w$ that is piecewise in $\Pp_k$, we have
\begin{equation*}
\| w-w_h \|_{V} \,\le C\Big(\|w-w_I\|_{V}+\|w-\wp\|_{h,V}+\|f-f_h\|_{V_h^\prime}
\Big)
\end{equation*}
where $C$ is a constant depending only on
$\alpha$, $\alpha_*$, $\alpha^*$, $M$ and, with the usual notation, the norm in
$V_h^\prime$ is defined as
\begin{equation}\label{defrhs}
\|f-f_h\|_{V^\prime_h}:=\,\sup_{v_h\in V_h}\frac{<f-f_h,v_h>}{\| v_h\|_{V}}.
\end{equation}
\end{theorem}

\subsection{Construction of the right-hand side}\label{load}
In order to build the loading term $<f_h,v_h>$ for $v_h\in V_h$ in a simple and easy way it is convenient to have internal degrees of freedom in $V_h$, and this means, according to \eqref{degrees-r-s-m} and \eqref{interior}, that $k\ge 4$ is needed.  In \cite{Brezzi:Marini:plates} suitable choices were made for different values of $k$, enough to guarantee the proper order of convergence in $H^2$. Namely,
\begin{equation}\label{stima-H2}
\| w - w_h \|_V \le C\, h^{k-1} \|w\|_{k+1}.
\end{equation}
In order to derive optimal estimates in $L^2$ and $H^1$ we need to make different choices. To this end, following \cite{projectors}, we modify the definition \eqref{def:Vh} of $V_h$. For $k\ge 2$,  and $r$ and $s$ related to $k$ by \eqref{degrees-r-s-m}, let $W^k(\E)$ be the new local space, given by
\begin{equation*}
W^k(\E):=\{v\in H^2(\E): ~v_{|e}\in\Pp_r(e),~\frac{\partial v}{\partial n}_{|e}\in \Pp_s(e)~\forall e\in \partial{\E},~\Delta^2 v\in \Pp_{k-2}(\E)\}. 
\end{equation*}
For $k=2$ we define the new global space as
\begin{equation}\label{def:enhanced-2}
W^2_h=\{v\in V: v_{|\E}\in W^2(\E),
\mbox{and }\int_{\E} v \, \dx = \int_{\E} \Pi^{\E}_k v \, \dx ~\forall \E \in \Th\},
\end{equation}
and for $k\ge 3$
\begin{equation}\label{def:enhanced}
\begin{array}{ll}
W^k_h&=\{v\in V: v_{|\E}\in W^k(\E),~\mbox{and}\\[3mm]
&\hskip1cm \int_{\E} v \,p_{\alpha} \dx = \int_{\E} \Pi^{\E}_k v \,p_{\alpha} \dx,~\alpha=k-3,k -2~\forall \E \in \Th\}.
\end{array}
\end{equation}
In \eqref{def:enhanced} $p_{\alpha}$ are homogeneous polynomials of degree $\alpha$.
It can be checked that the d.o.f. \eqref{vertices}--\eqref{interior} are the same, but the added conditions on the moments allow now to compute the $L^2-$projection of any $v\in W^k_h$ onto $\Pp_{k-2}(\E)~\forall \E$, and not only onto $\Pp_{k-4}(\E)$ as before. Taking then $f_h=L^2-$projection of $f$ onto the space of piecewise polynomials of degree $k-2$, 
that is,
$$
f_h = \Pi^0_{k-2} f \ \qquad \mbox{on each }
\E \in\Th ,
$$
the right-hand side in \eqref{Vem-pb} can be exactly computed:
\begin{equation*}
\begin{aligned}
<f_h,v_h> &= \sum_{{\E}\in\Th} \int_{\E} f_h\, v_h \dx\equiv
\sum_{{\E}\in\Th} \int_{\E} (\Pi^0_{k-2}f)\, v_h \dx\\
&=
\sum_{{\E}\in\Th} \int_{\E} f\, (\Pi^{0}_{k-2}v_h) \dx .
\end{aligned}
\end{equation*}
Moreover, standard $L^2$ orthogonality and approximation estimates 
yield
\begin{equation}\label{estk2}
\begin{array}{ll}
<f_h,v_h>-(f, v_h) &= \displaystyle{\sum_{\E\in\Th} \int_{\E} (\Pi^{0}_{k-2}f-f) (v_h - \Pi_{k-2}^{0} v_h) \dx}\\
&\displaystyle{\le C \sum_{\E\in\Th}h_{\E}^{k-1} |f|_{k-1,\E} \,\|v_h - \Pi_{k-2}^{0} v_h\|_{0,\E}.}
\end{array}
\end{equation}

\section{Estimate in $H^1$}\label{est-H1}
We shall use duality arguments, both for deriving estimates in $H^1$ and in $L^2$. In view of this,  let us recall some regularity results   for the problem
\begin{equation}\label{dual-general}
D \Delta^2 \psi = g \qquad \mbox{in } \O, \quad \psi=\psi_{/n}=0\qquad\mbox{on } \partial \O .
\end{equation}
Since $\O$ is a convex polygon, it holds (see \cite{Grisvard})
\begin{equation}\label{regularity-1}
g\in H^{-1}(\O) \Longrightarrow~\psi\in H^3(\O),  \quad \|\psi\|_{3} \le C\, \|g\|_{-1},
\end{equation}
and
\begin{equation}\label{regularity-2}
\begin{aligned}
&\exists\,  \reg \mbox{ with } 0<\reg\le1 \mbox{ such that}\\
&g\in L^{2}(\O) \Longrightarrow~\psi\in H^{3+\reg}(\O),  \quad\|\psi\|_{3+\reg} \le C\, \|g\|_{0} .
\end{aligned}
\end{equation}
The value of $\reg$ depends on the maximum angle in $\O$. Moreover, there exists a $\theta_0 <\pi$ such that, for all $\theta \le \theta_0$ it holds $\reg=1$, thus giving $\psi \in H^4(\O)$.

We shall prove the following result.
\begin{theorem}\label{stima-in-H1}
Let $w$ be the solution of \eqref{pb-cont}, and let $w_h$ be the solution of \eqref{Vem-pb}. Then
\begin{equation}\label{stima-H1}
|w-w_h|_1 \le C h^k ( |w|_{k+1} +(\sum_{\E\in\Th} |f|_{k-1,\E}^2)^{1/2}),
\end{equation}
with $C$ a positive constant independent of $h$.
\end{theorem}
\begin{proof}
Let $\psi \in H^2_0(\Omega)$ be the solution of \eqref{dual-general} with $g=-\Delta (w-w_h)$:
\begin{equation}\label{dual}
D \Delta^2 \psi = -\Delta (w-w_h) \qquad \mbox{in } \O.
\end{equation}
By \eqref{regularity-1} we have
\begin{equation}\label{stima-reg3}
\|\psi\|_{3} \le C\, \|\Delta (w-w_h)\|_{-1}\le C\,  |w-w_h|_1 .
\end{equation}
Let $\psi_I$ be the interpolant of $\psi$ in $W^2_h$, for which it holds (see \eqref{errinterp})
\begin{equation}\label{stima-interp-psi}
\|\psi -\psi_I\|_m \le C\, h^{3 - m} \|\psi\|_{3},~~m=0,1,2.
\end{equation}
Integrating by parts, using \eqref{dual}, adding and subtracting $\psi_I$, and using \eqref{pb-cont} and \eqref{Vem-pb} we have:
\begin{equation}\label{per-stima-H1}
\begin{array}{ll}
\!\!\!\!\!\!\!\!|w-w_h|^2_1 &= -(\Delta (w-w_h), w-w_h)_0= (D \Delta^2 \psi, w-w_h)_0\\[2mm]
&= a(w-w_h, \psi -\psi_I) + a(w-w_h,\psi_I)\\[2mm]
&=a(w-w_h, \psi -\psi_I)+ [(f,\psi_I)- <f_h,\psi_I> ]\\[2mm]
&+ [a_h(w_h,\psi_I)-a(w_h,\psi_I)] =: T_1 + T_2 + T_3.
\end{array}
\end{equation}
The first term is easily bounded through \eqref{cont-a}, and then \eqref{stima-H2},  \eqref{stima-interp-psi}, and \eqref{stima-reg3}:
\begin{equation}\label{term-T1}
T_1 
\le C h^{k-1} \|w\|_{k+1} h  \|\psi\|_3 
\le C h^{k} \|w\|_{k+1} |w-w_h|_1.
\end{equation} 
For $T_2$ we use \eqref{estk2} with $v_h=\psi_I$. Standard interpolation estimates give
\begin{equation*}
\begin{array}{ll}
\|\psi_I-\Pi^0_{k-2}\psi_I\|_{0,\E}& \le \|\psi_I-\Pi^0_{0}\psi_I\|_{0,\E} \\[2mm]
&\le \|\psi_I-\psi\|_{0,\E} +\|\psi-\Pi^0_{0}\psi\|_{0,\E}+\|\Pi^0_{0}(\psi-\psi_I)\|_{0,\E}\\[2mm]
&\le C h_{\E}\, |\psi|_{1,\E} 
\end{array}
\end{equation*} 
which inserted in \eqref{estk2} gives
\begin{equation}\label{term-T2}
T_2 \le C \,h^{k}(\sum_{\E\in\Th} |f|_{k-1,\E}^2)^{1/2} \,  |w-w_h |_{1}.
\end{equation} 
It remains to estimate $T_3$. Adding and subtracting $\wp(=\Pi^0_k w)$ and using \eqref{basic}, then adding and subtracting $\psi_{\pi}=\Pi^0_2{\psi}$ and using again \eqref{basic} we have 
\begin{equation*}\label{term-T3}
\begin{array}{lll}
T_3&=&\displaystyle{\sum_{\E} (a^{\E}_h(w_h,\psi_I)-a^{\E}(w_h,\psi_I))}\\
&=&\displaystyle{\sum_{\E} (a^{\E}_h(w_h-\wp,\psi_I)+a^{\E}(\wp-w_h,\psi_I))}\\
&=&\displaystyle{\sum_{\E} (a^{\E}_h(w_h-\wp,\psi_I-\psi_{\pi})+a^{\E}(\wp-w_h,\psi_I-\psi_{\pi}))}.
\end{array}
\end{equation*}
From \eqref{CSahEO}, \eqref{cont-a}, standard approximation estimates, \eqref{stima-H2} and \eqref{stima-interp-psi} we deduce
\begin{equation}\label{term-T3}
T_3 \le C\, h^k  |w|_{k+1}|w-w_h|_1.
\end{equation}
Inserting \eqref{term-T1}, \eqref{term-T2}, and \eqref{term-T3} in \eqref{per-stima-H1} we have the result \eqref{stima-H1}
\end{proof}
\section{Estimate in $L^2$}\label{est-L2}
We shall prove the following result.
\begin{theorem}\label{stima-in-L2}
Let $w$ be the solution of \eqref{pb-cont}, and let $w_h$ be the solution of \eqref{Vem-pb}. Then
\begin{equation}\label{stima-L2}
\|w-w_h\|_0 \le C \begin{cases}& \displaystyle{h^{2} \Big( |w|_{3} +(\sum_{\E\in\Th} |f|_{1,\E}^2)^{1/2}\Big) ~~\mbox{for } k=2}\\
&\displaystyle{h^{k+\reg} \Big( |w|_{k+1} +(\sum_{\E\in\Th} |f|_{k-1,\E}^2)^{1/2}\Big) ~~\mbox{for } k\ge 3},
\end{cases}
\end{equation}
with $C>0$ independent of $h$, and $\reg$ the regularity index given in \eqref{regularity-2}.
\end{theorem}
\begin{proof}
We shall treat only the cases $k\ge 3$, the reason being that, if $\beta$ is the order of convergence in $H^2$, the expected order in $L^2$ is given by $\min\{2\beta, \beta+2\}$. Hence, for $k=2$ we can expect not more than order $2$, which is a direct consequence of the $H^1-$estimate \eqref{stima-H1}: 
\begin{equation}
\mbox{for } k=2, \quad \|w-w_h\|_0\le C\,|w-w_h|_1 \le C h^2 \Big( |w|_{3} +(\sum_{\E\in\Th} |f|_{1,\E}^2)^{1/2}\Big).
\end{equation}
 Let then $k\ge 3$, and let $\psi \in H^2_0(\Omega)$ be the solution of \eqref{dual-general} with $g=w-w_h$:
\begin{equation}\label{dual-L2}
D \Delta^2 \psi = w-w_h \qquad \mbox{in } \O.
\end{equation}
By the regularity assumption \eqref{regularity-2} we have
\begin{equation}\label{reg-psi-L2}
\|\psi\|_{3+\reg} \le C\,  \|w-w_h\|_0 .
\end{equation}
Let $\psi_I$ be the interpolant of $\psi$ in $W^3_h$, for which it holds
\begin{equation}\label{interp-3}
\|\psi -\psi_I\|_m \le C\, h^{3+\reg - m} \|\psi\|_{3+\reg},~~m=0,1,2.
\end{equation}
Then, from \eqref{dual-L2} and proceeding as we did in \eqref{per-stima-H1} we have
\begin{equation}\label{per-stima-L2}
\begin{array}{ll}
\|w-w_h\|^2_0&=(D \Delta^2 \psi ,w-w_h)_0= a(\psi,w-w_h)\\[2mm]
&=a(w-w_h, \psi -\psi_I)+ [(f,\psi_I)- <f_h,\psi_I> ]\\[2mm]
&+ [a_h(w_h,\psi_I)-a(w_h,\psi_I)] =: T_1 + T_2 + T_3.
\end{array}
\end{equation}
The rest of the proof follows exactly the steps used for proving Theorem \ref{stima-in-H1}. 
Thus, from \eqref{stima-H2}, \eqref{interp-3} and \eqref{reg-psi-L2},
\begin{equation}\label{term-T1-L2}
T_1
\le C h^{k-1} \|w\|_{k+1} h^{1+\reg}  \|\psi\|_{3+\reg} 
\le C h^{k+\reg} \|w\|_{k+1} \|w-w_h\|_0.
\end{equation} 
For the term $T_2$ we use  again \eqref{estk2} with $v_h=\psi_I$, which now gives
\begin{equation*}
\|\psi_I-\Pi^0_{k-2}\psi_I\|_{0,\E} \le \|\psi_I-\Pi^0_{1}\psi_I\|_{0,\E} 
\le C h^2_{\E}\, |\psi|_{2,\E} ,
\end{equation*} 
so that
\begin{equation}\label{term-T2-L2} 
T_2 
\le C \,h^{k+1} \displaystyle{(\sum_{\E\in\Th} |f|_{k-1,\E}^2)^{1/2} \,  \|w-w_h \|_{0}} .
\end{equation}
Finally, proceeding exactly as for \eqref{term-T3},
\begin{equation}\label{term-T3-L2}
T_3
 \le C\, h^{k-1} |w|_{k+1} h^{1+\reg} \|\psi\|_{3+\reg} 
\le C\, h^{k+\reg}  |w|_{k+1}\|w-w_h\|_0.
\end{equation} 
Collecting \eqref{term-T1-L2}--\eqref{term-T3-L2}  in \eqref{per-stima-L2} gives
\begin{equation*}
\|w-w_h\|_0 \le C\,h^{k+\reg} ( |w|_{k+1} +(\sum_{\E\in\Th} |f|_{k-1,\E}^2)^{1/2})
\end{equation*}
and the proof is concluded.
\end{proof}
\section{Numerical results}\label{num-tests}
In order to assess accuracy and  performance of virtual elements for plates, we present numerical tests using the first two elements of the family here described. 
The corresponding 
 polynomial degree indices, defined in \eqref{degrees-r-s-m},  are $r=3$, $s=1$, $m=-2$ and $r=3$, $s=2$, $m=-1$. Thus, the elements are named VEM31 and VEM32, respectively.
 The degrees of freedom, chosen according to the definitions \eqref{vertices}-\eqref{interior}, are  the values of the displacement and its first derivatives at the vertices  (\eqref{vertices} and  \eqref{gradvertices}) for VEM31, and  the same degrees of freedom \eqref{vertices}-\eqref{gradvertices} plus the moment of order zero of the normal derivative (see \eqref{interior}) for VEM32.
The two elements  are presented in Figure \ref{fig:vem}.
They are the extensions to polygonal elements of two well-known finite elements for plates: the Reduced Hsieh-Clough-Tocher triangle (labelled CLTR),  and 
the Hsieh-Clough-Tocher  triangle (labelled CLT) (see e.g. \cite{Ciarlet-78}), respectively.
As a test problem we solve \eqref{biharm}-\eqref{BC} with $\Omega=$ unit square and $f$ chosen to have as exact solution the function $w_{ex}=x^2(x-1)^2y^2(y-1)^2$.
As a first test, we compare the behaviour of virtual and finite elements; for this we take a sequence of uniform meshes of $N\times N\times 2$ equal right triangles ($N=4,8,16,32$), and
we plot the convergence curves of the error in $L^2$, $H^1$ and $H^2$ 
produced by the virtual elements  VEM31 and VEM32,  and the finite elements CLTR and CLT respectively.
Figure \ref{fig:errorVEM-FEM} shows the relative errors in $L^2,~H^1$ and $H^2$ norm against  the mesh diameter ($h=0.3536,~h=0.1768,~h=0.0884,~h=0.0442$). The convergence rates are obviously the same, although in all cases the virtual elements seem to perform a little better.
Next, we test the behaviour of the virtual elements on a sequence of random Voronoi polygonal tessellations of the unit square in  25, 100, 400, 1600 polygons with mean diameters $h=0.3071, ~h=0.1552, ~h=0.0774, ~h=0.0394$ respectively. (Figure \ref{fig:mesh}  shows the  100 and 1600-polygon meshes). 
In Figure \ref{fig:errorVEMv} we compare the convergence curves in $L^2,~H^1$ and $H^2$ norm   obtained using the virtual elements VEM31 and VEM32 on the Voronoi meshes and  on uniform triangular meshes.
\begin{figure}[ht!]
\centering
\includegraphics[width=3.8cm]{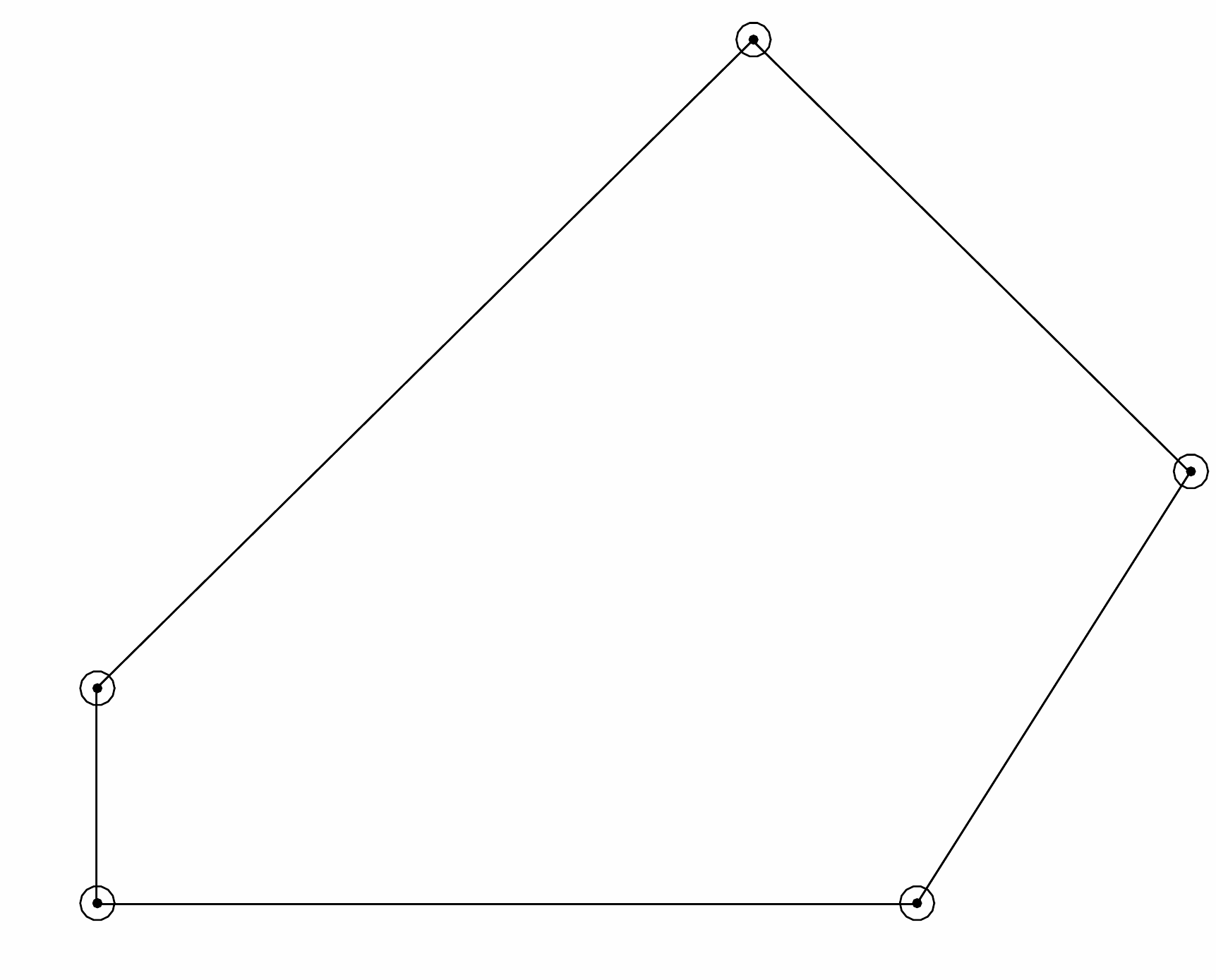}
\hskip1.cm
\includegraphics[width=3.8cm]{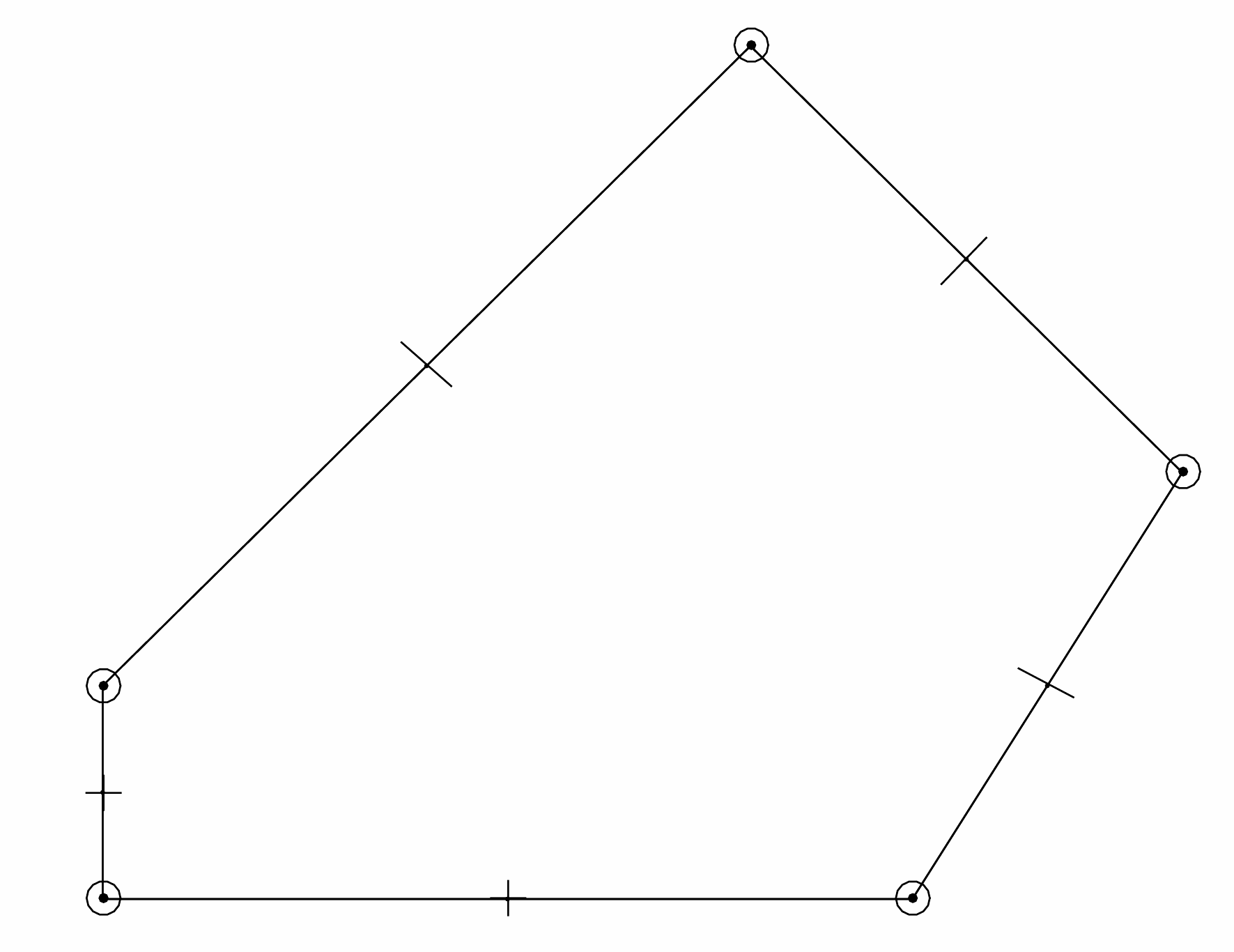}
\caption{VEM31 element on the left, VEM32 element on the right}
\label{fig:vem}
\centering
\includegraphics[width=5cm,height=4.3cm]{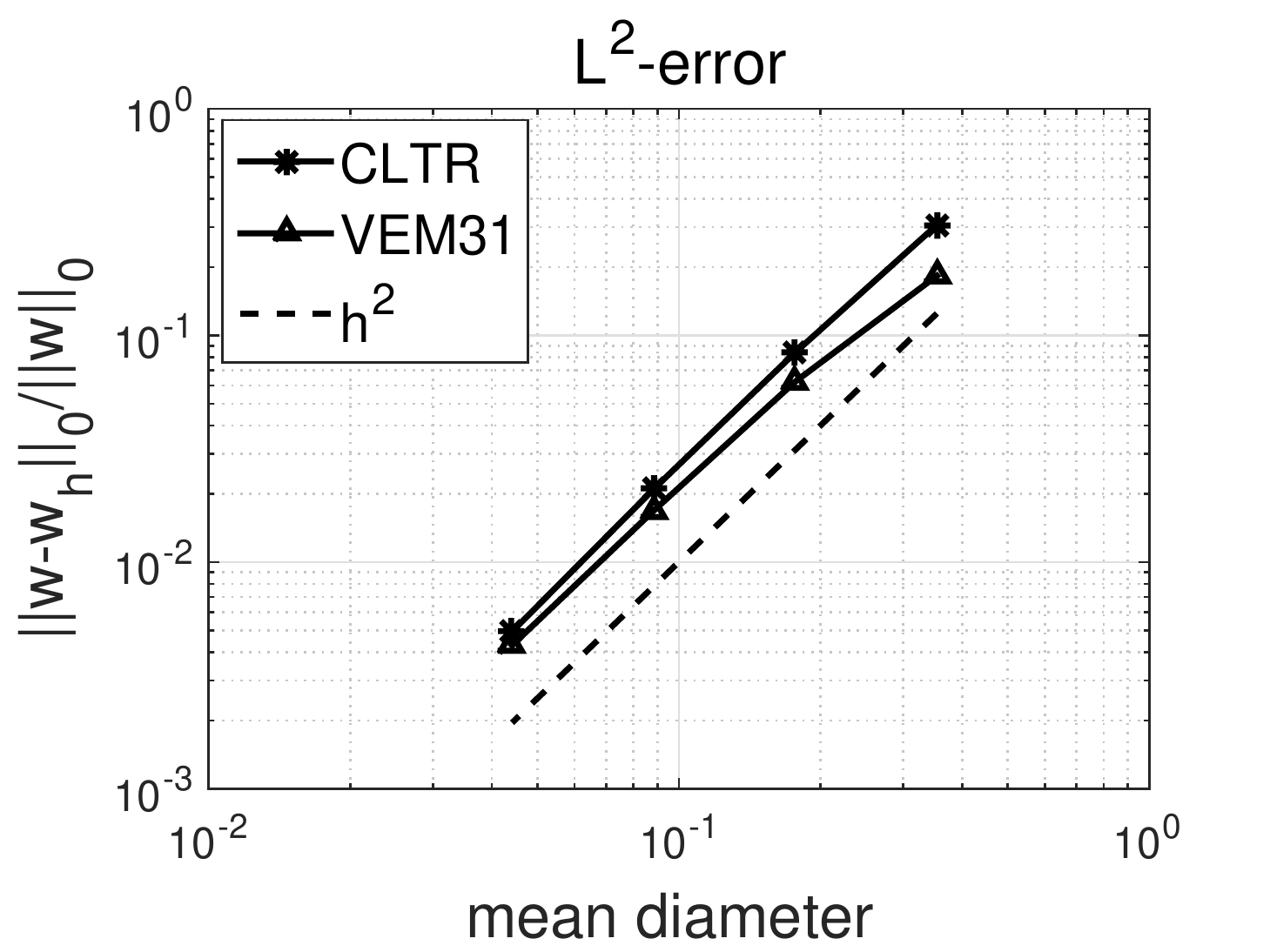}
\hskip0.3cm
\includegraphics[width=5cm,height=4.3cm]{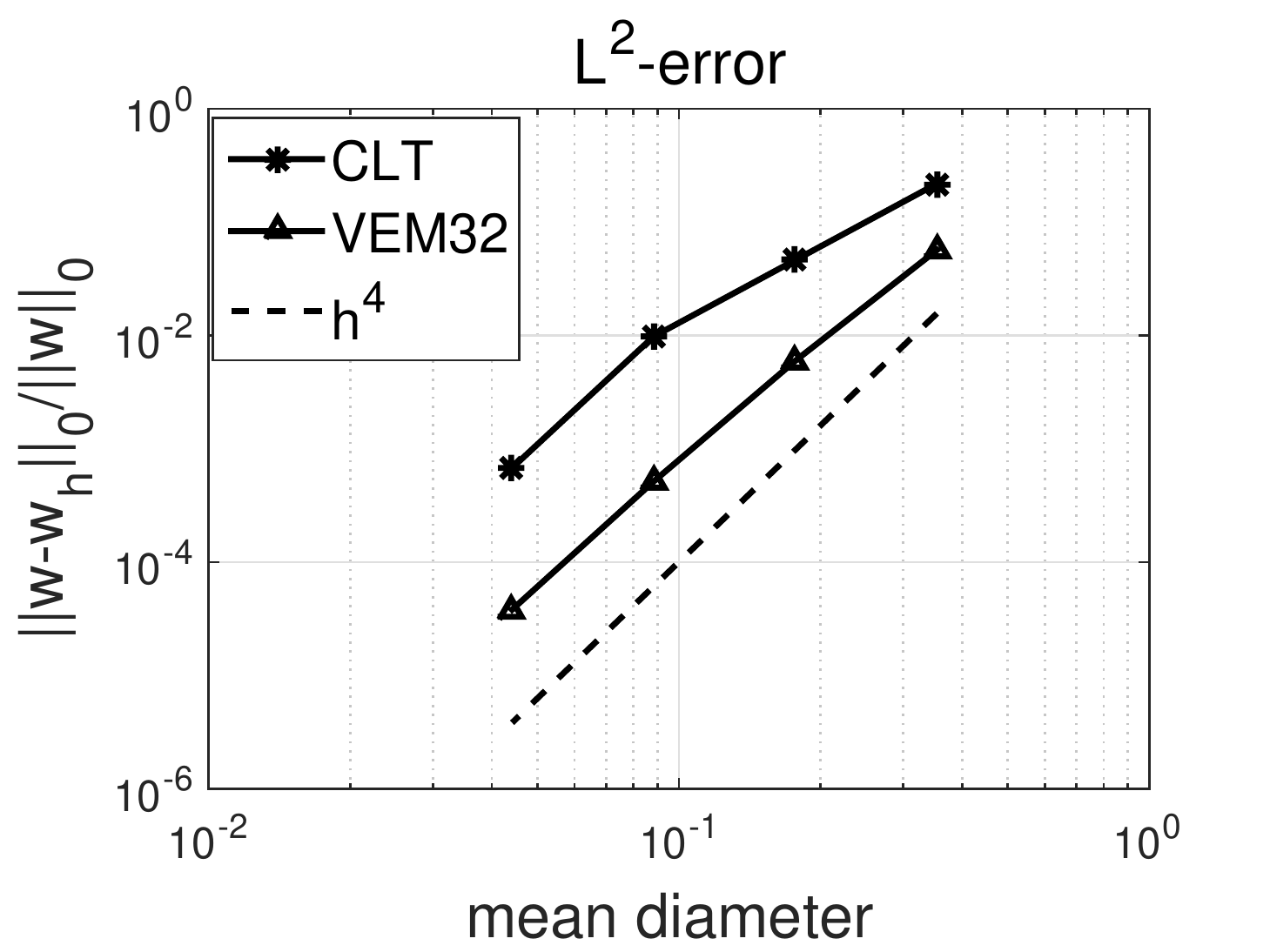}
\vskip0.3cm
\includegraphics[width=5cm,height=4.3cm]{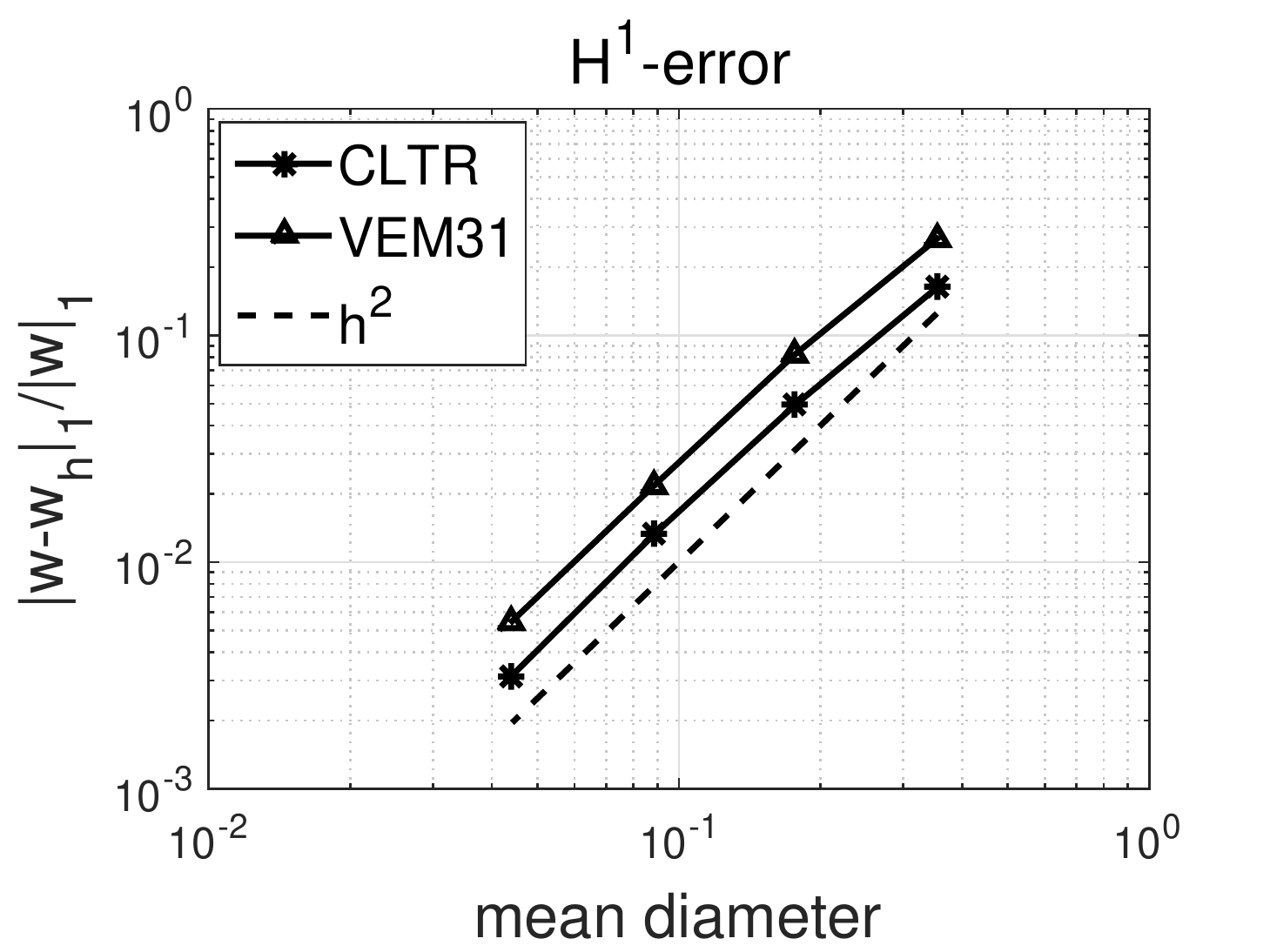}
\hskip0.3cm
\includegraphics[width=5cm,height=4.3cm]{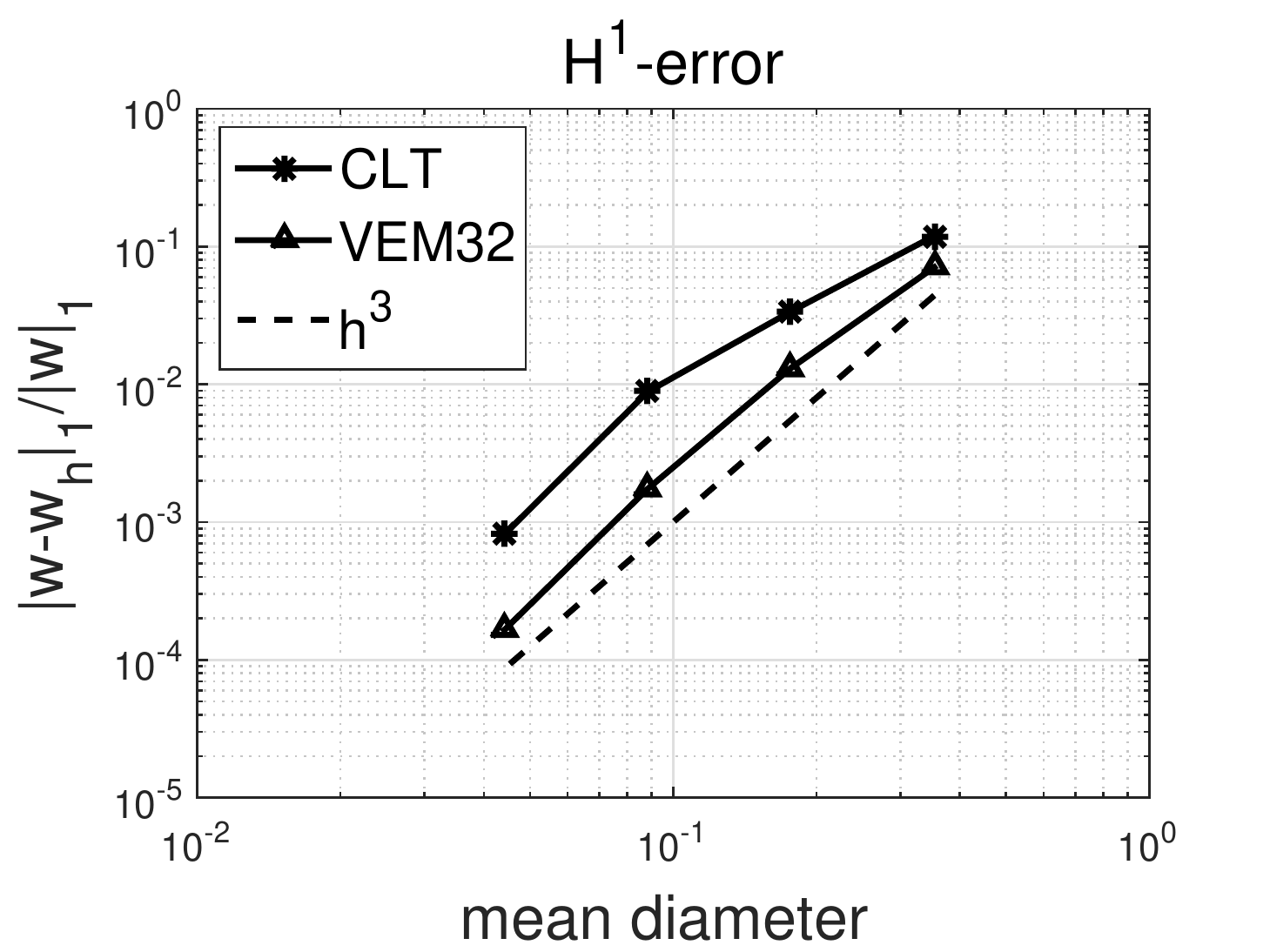}
\vskip0.3cm
\includegraphics[width=5cm,height=4.3cm]{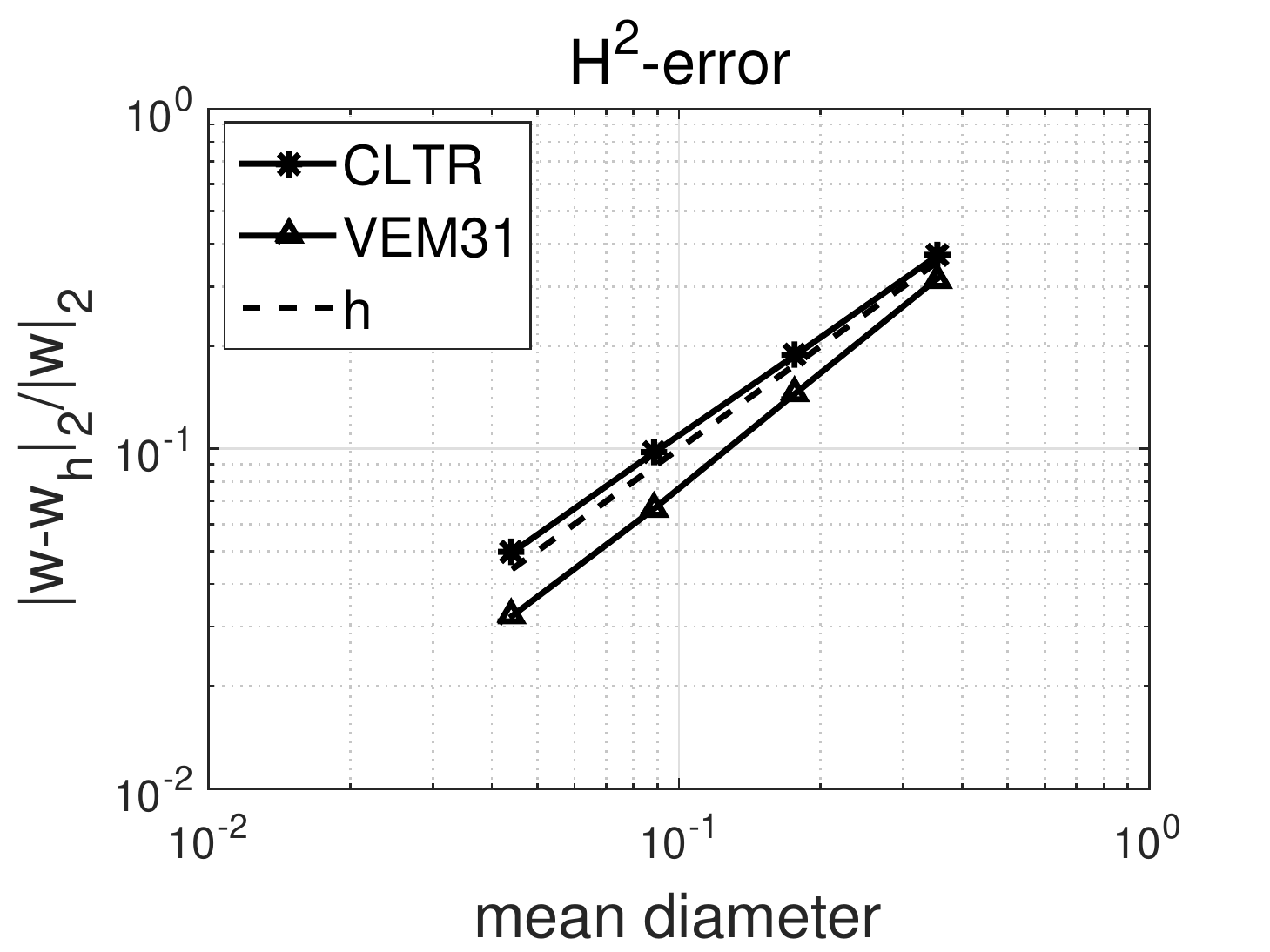}
\hskip0.3cm
\includegraphics[width=5cm,height=4.3cm]{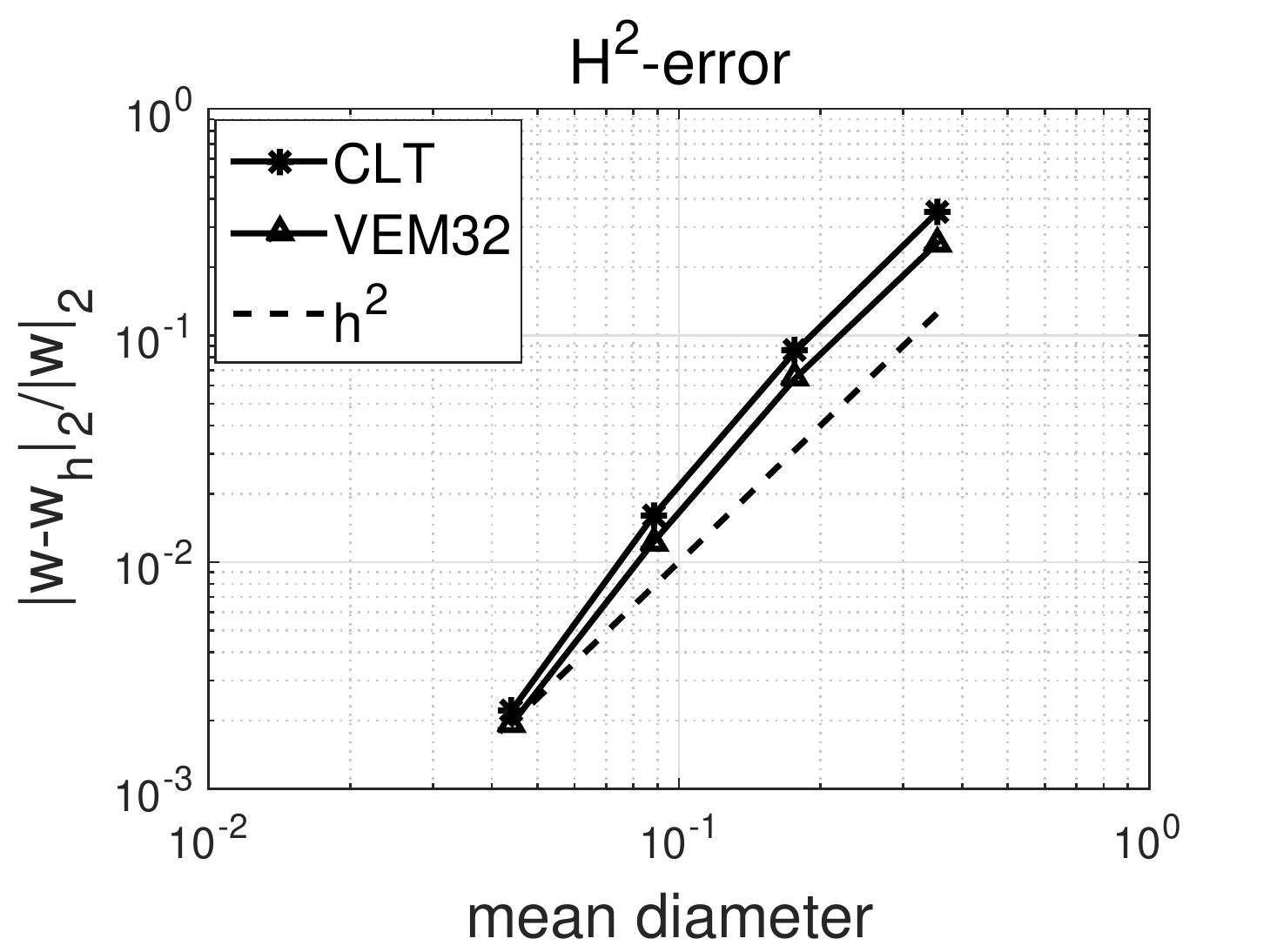}
\caption{Virtual elements compared with the corresponding finite elements. Left: VEM31 and CLTR. Right: VEM32 and CLT }
\label{fig:errorVEM-FEM}
\end{figure}
\begin{figure}[ht!]
\centering
\includegraphics[width=3.5cm]{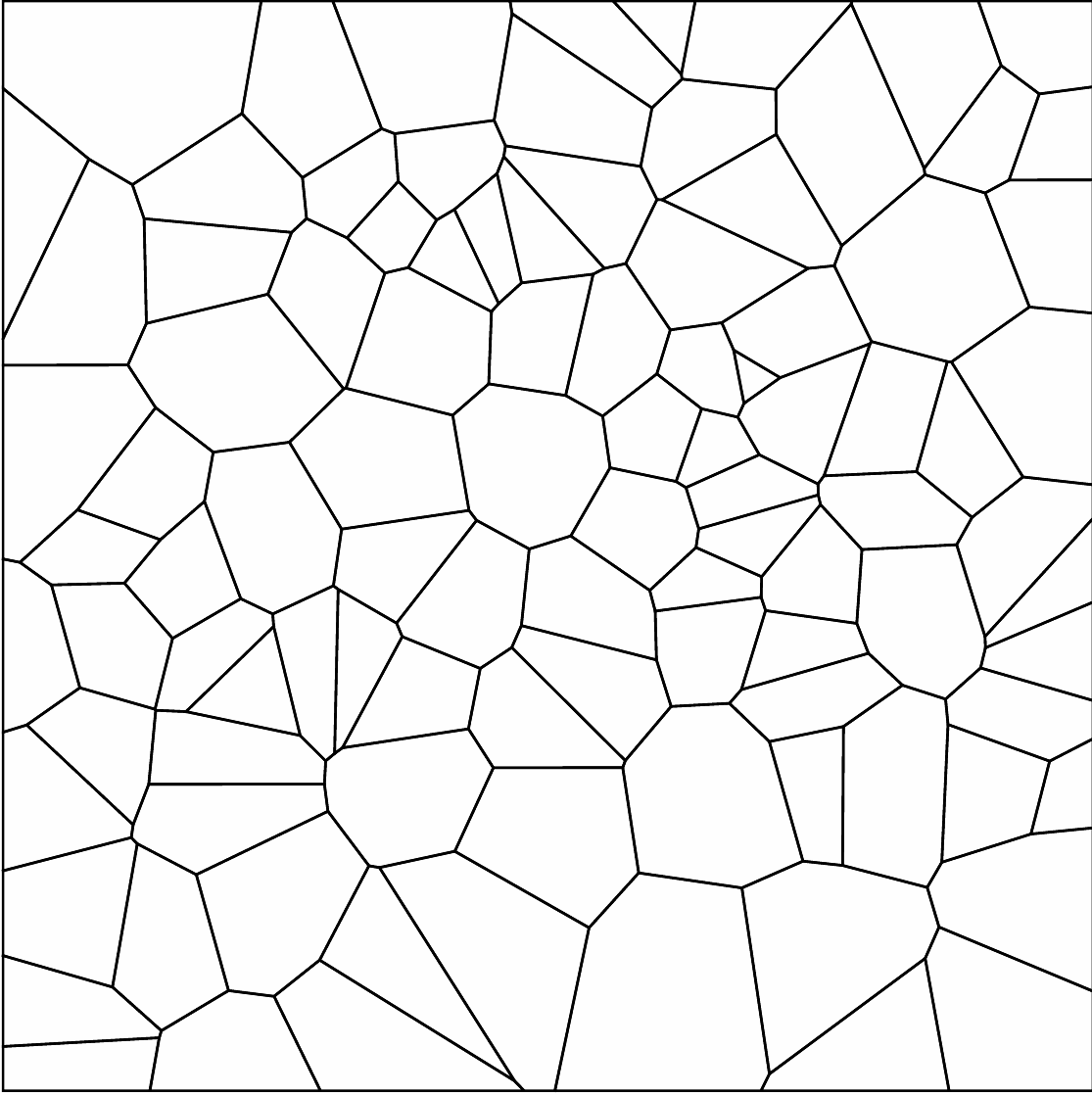}
\hskip1.cm
\includegraphics[width=5.5cm]{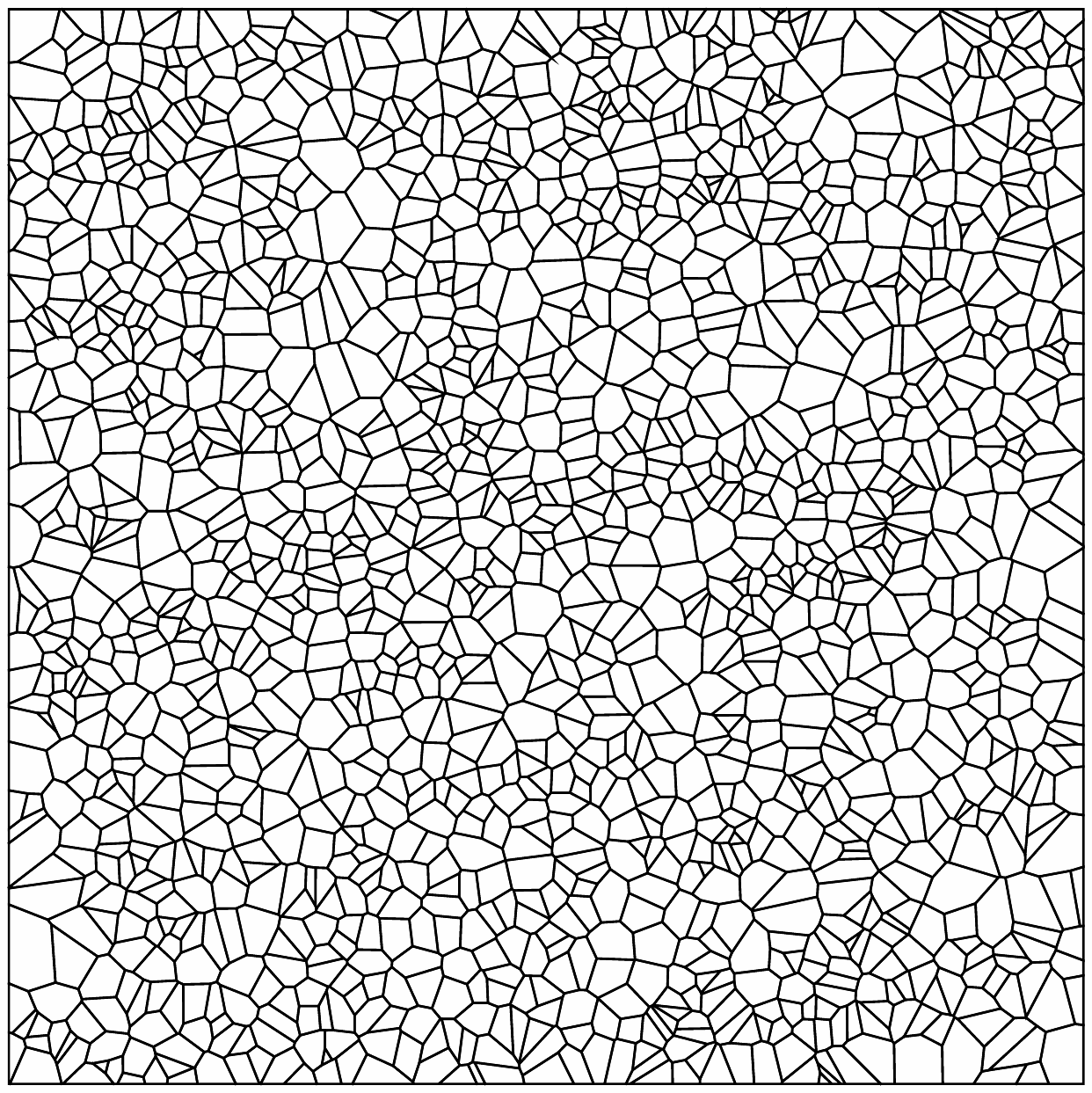}
\caption{100 and 1600-polygons  Voronoi mesh}
\label{fig:mesh}
%
%
\centering
\includegraphics[width=5cm,height=4.3cm]{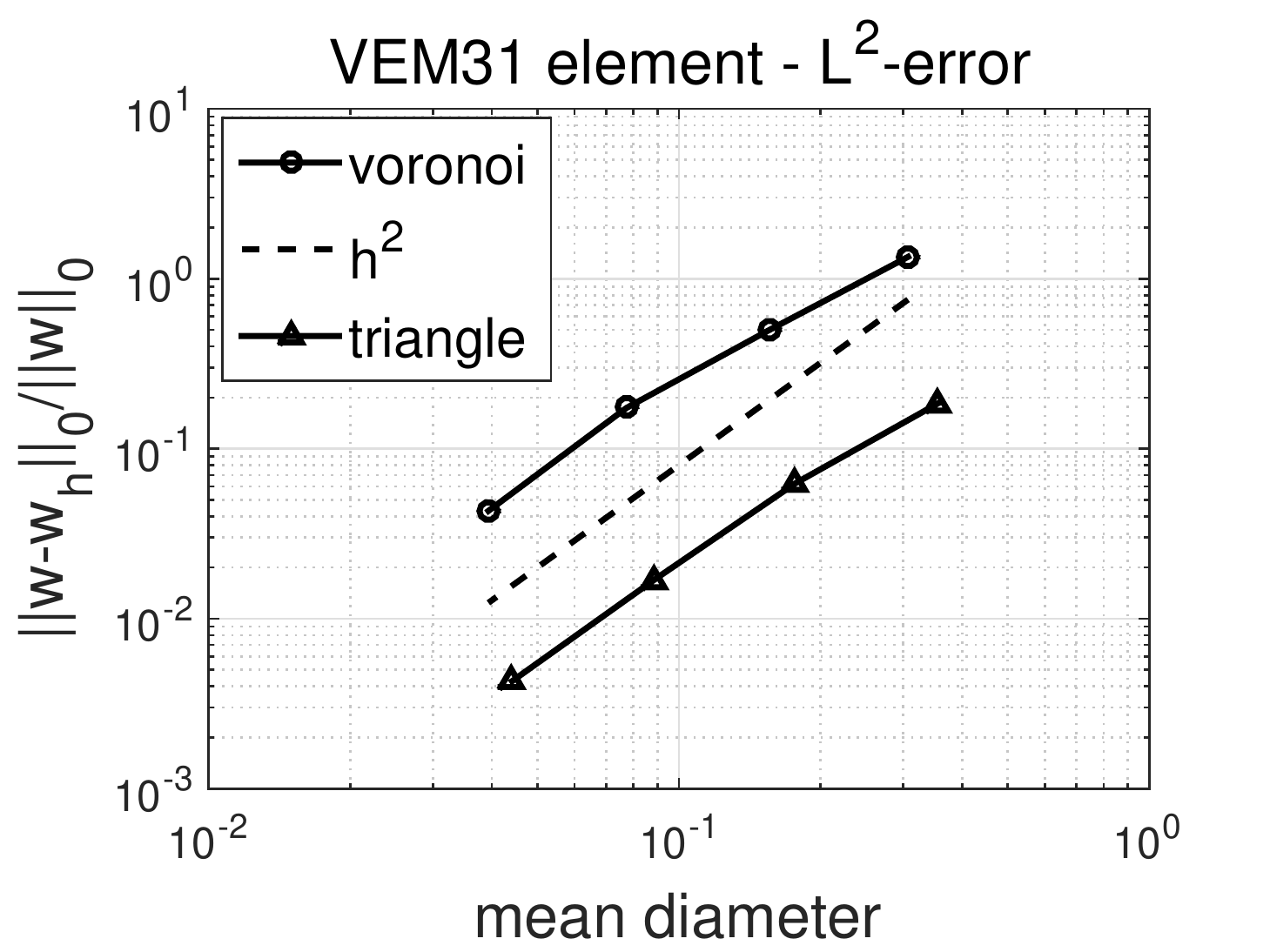}
\hskip0.3cm
\includegraphics[width=5cm,height=4.3cm]{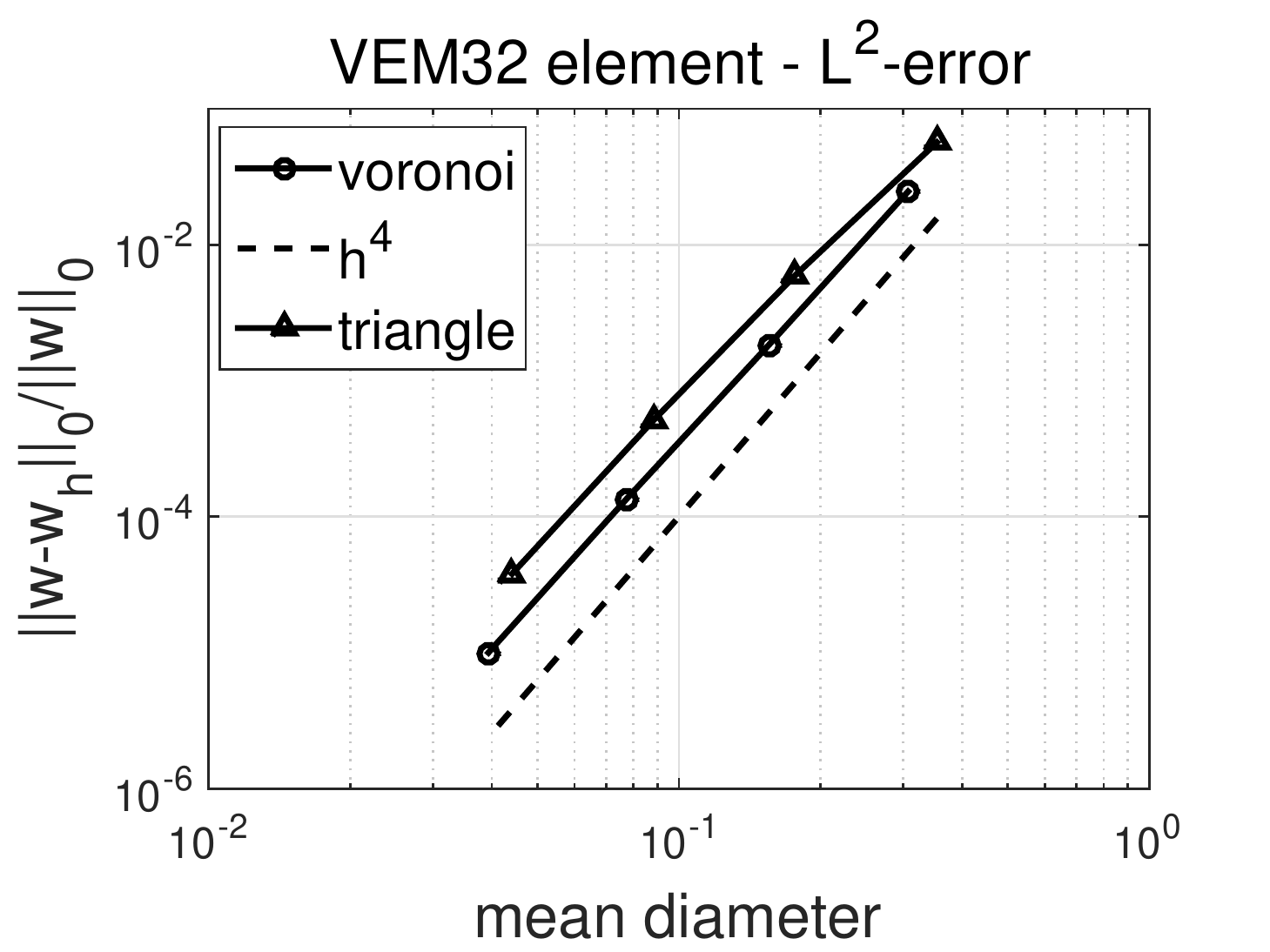}
\vskip0.3cm
\includegraphics[width=5cm,height=4.3cm]{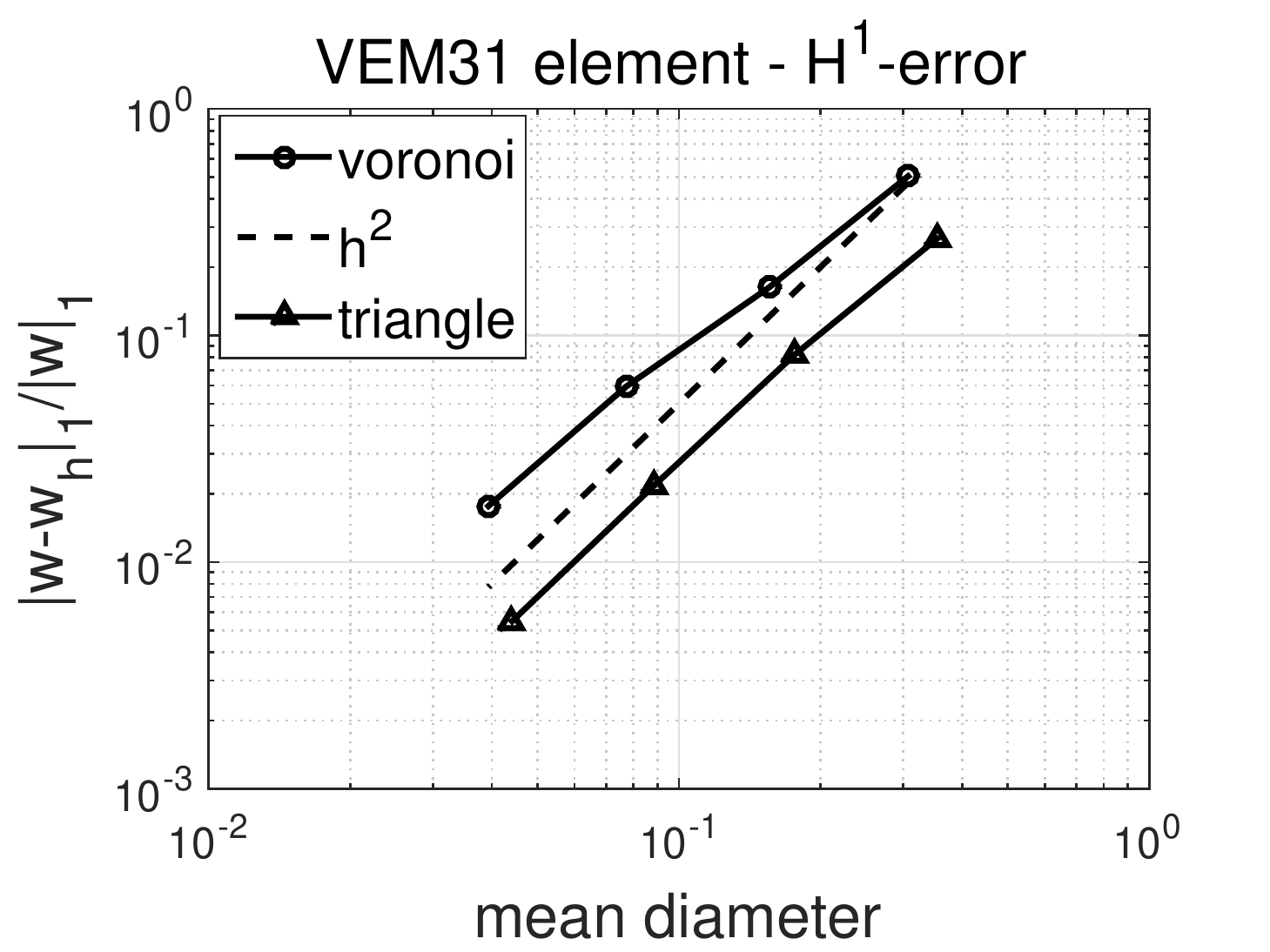}
\hskip0.3cm
\includegraphics[width=5cm,height=4.3cm]{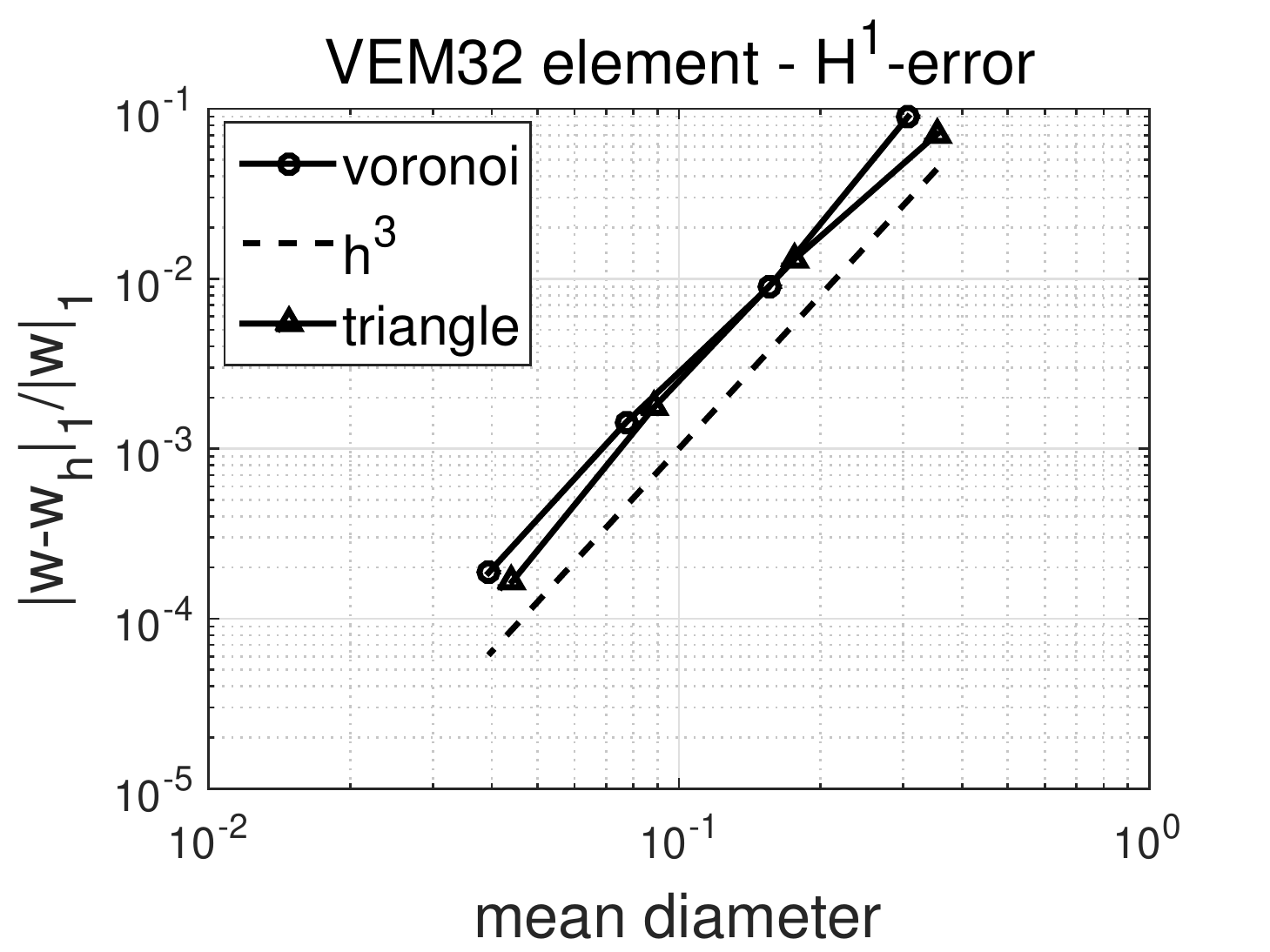}
\vskip0.3cm
\includegraphics[width=5cm,height=4.3cm]{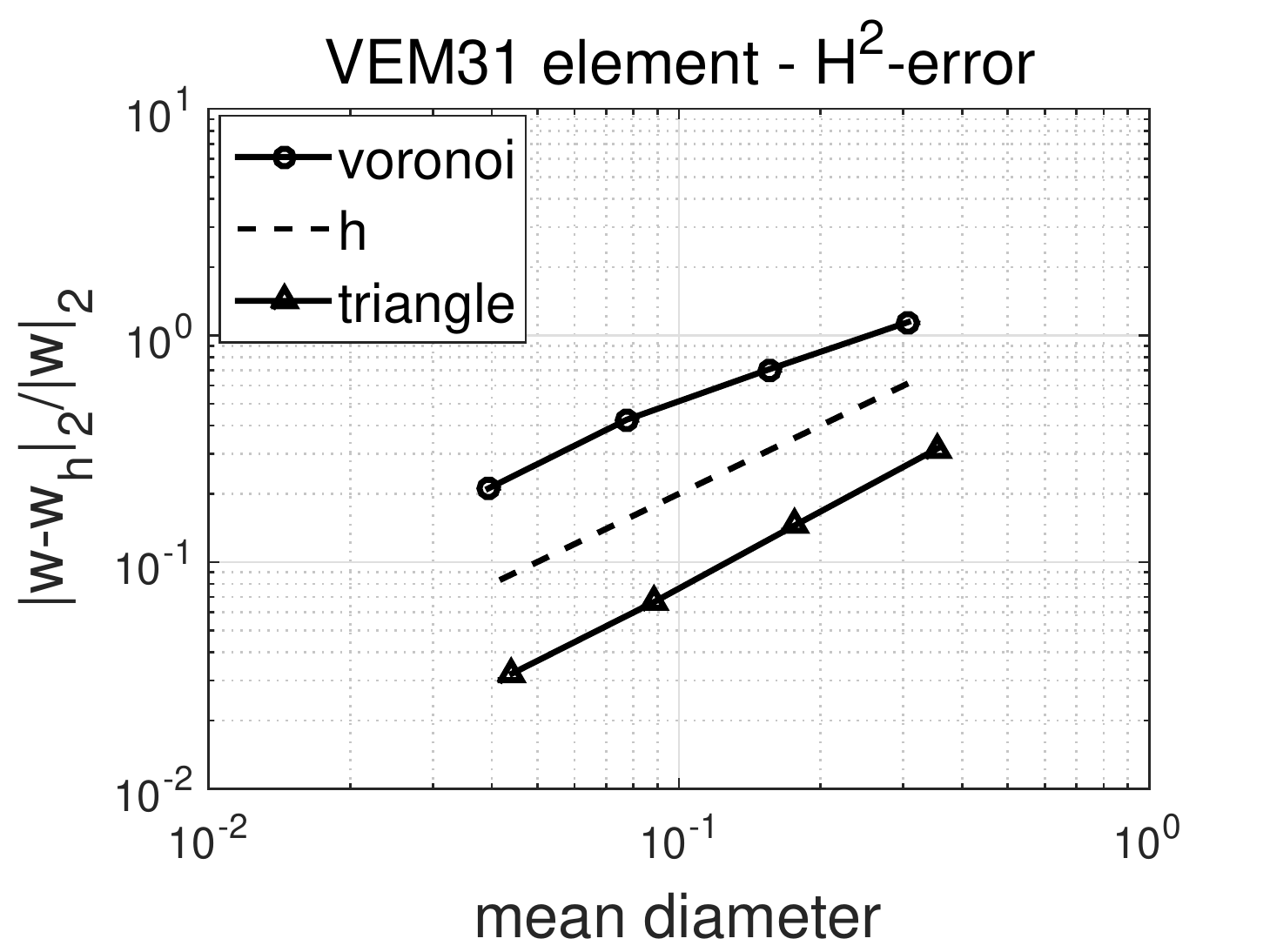}
\hskip0.3cm
\includegraphics[width=5cm,height=4.3cm]{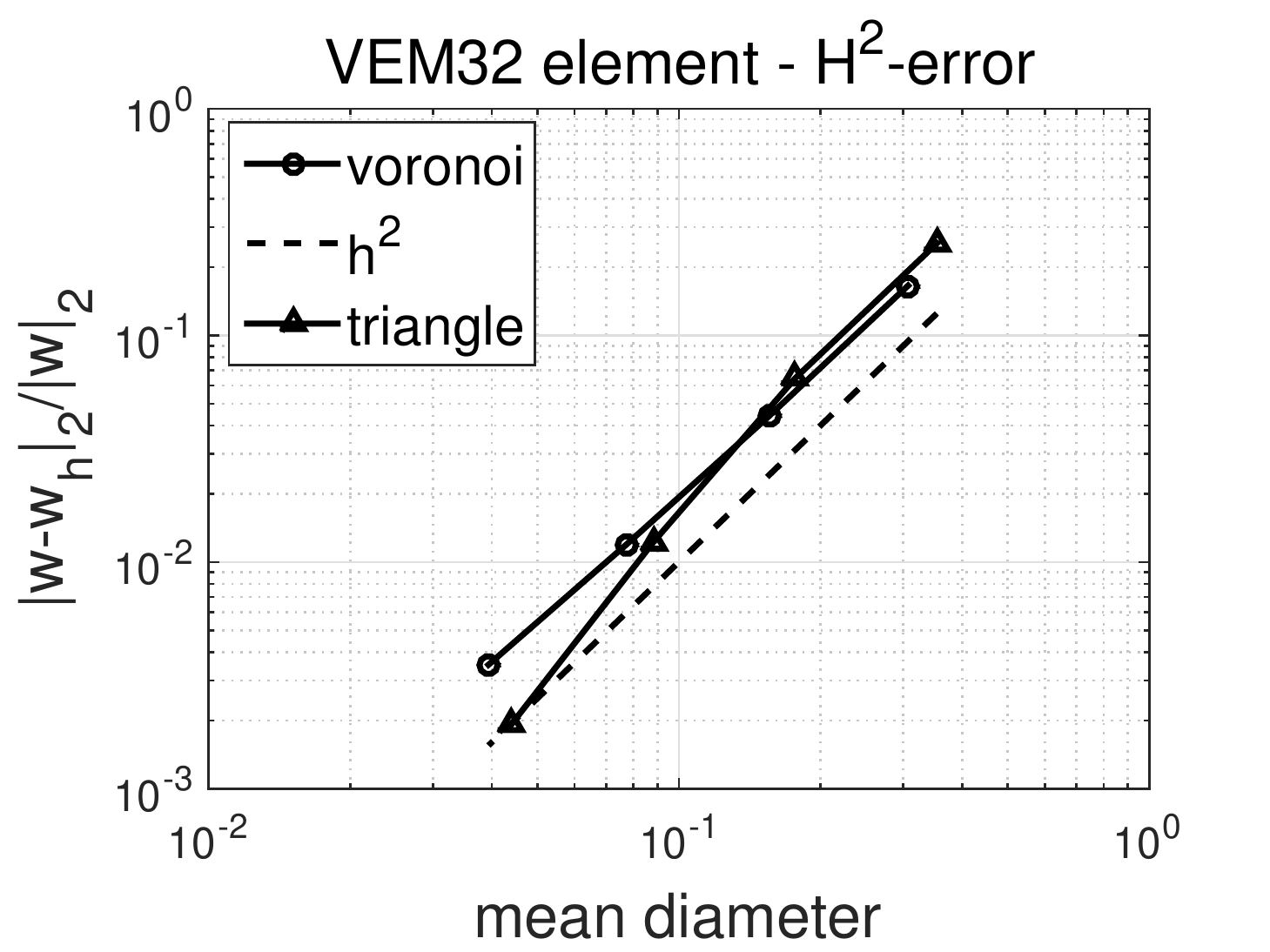}
\caption{Virtual elements on different meshes. Left: VEM31 element. Right: VEM32 element}
\label{fig:errorVEMv}
\end{figure}
\vfill
\eject

\bibliographystyle{amsplain}

\bibliography{general-bibliography}

\end{document}